\documentclass[11pt,reqno]{amsart}
\usepackage{amsmath,amssymb,mathtools}
\usepackage{enumitem}
\usepackage[hidelinks]{hyperref}
\usepackage[margin=1in]{geometry}

\newtheorem{theorem}{Theorem}[section]
\newtheorem{proposition}[theorem]{Proposition}
\newtheorem{lemma}[theorem]{Lemma}

\theoremstyle{remark}
\newtheorem{remark}[theorem]{Remark}

\newcommand{\supp}{\operatorname{supp}}
\newcommand{\Homeo}{\operatorname{Homeo}}
\newcommand{\Fix}{\operatorname{Fix}}
\numberwithin{equation}{section}

\title[ORDER AUTOMORPHISM OF A DLAB GROUP]{AN ORDER AUTOMORPHISM OF A DLAB GROUP NOT INDUCED BY CONJUGATION}
\author{Ting Gong}
\address{Department of Mathematics, University of Washington, Seattle, WA 98195}
\email{tgong2@uw.edu}

\author{Yong Yang}
\address{Department of Mathematics, Texas State University, San Marcos, TX 78666}
\email{yang@txstate.edu}

\author{Michael Ruofan Zeng}
\address{Department of Mathematics, University of Washington, Seattle, WA 98195}
\email{zengrf@uw.edu}
\date{}

\begin{document}
\begin{abstract}
Let $G=D_{\langle2\rangle}([0,1])$ be equipped with either of its two Dlab orders.  There exists an order automorphism $\alpha$ of $G$ such that, for every rank-one subgroup $H\leq\mathbb R_{>0}^{\times}$, every one of the six corresponding Dlab groups $A$ listed below, equipped with any of its linear orders, every order-preserving embedding $e:G\hookrightarrow A$, and every $u\in A$, there exists $f\in G$ such that $e(\alpha(f))\neq u^{-1}e(f)u$.
\end{abstract}

\subjclass[2020]{06F15, 20F60, 20E36}
\keywords{ordered group, Dlab group, order automorphism, piecewise linear homeomorphism, conjugation}
\maketitle

\section{Introduction}
An order automorphism of an ordered group is an automorphism preserving its specified linear order.  Kopytov and Medvedev posed Problem~21.149 in the \emph{Kourovka Notebook}; in its original form, the problem asks whether a Dlab group can have an order automorphism that is not inner \cite{KN}.  Van Doorn, Judin, Monticone and Morrison report that they obtained an affirmative solution to this original formulation and sent it to the editors of the \emph{Kourovka Notebook}.  According to their account, one of the problem authors checked the solution and agreed that it was correct, but then explained that the intended question was stronger: whether there exists an order automorphism of a Dlab group that is not induced by conjugation by an element of a possibly larger Dlab group \cite[Appendix~A]{vDJM}.  They further state that the problem statement was revised to reflect this intended meaning and that the stronger problem remained open.  Thus the distinction is essential: an automorphism may be non-inner in the original Dlab group but still be induced by conjugation after that group is embedded in a larger Dlab group.  The theorem below addresses this clarified form of Problem~21.149 for the six rank-one Dlab groups considered here.

Put
\[
I=[0,1],\qquad G=D_{\langle2\rangle}(I).
\]
For a homeomorphism $g$, write $\supp(g)=\{x:g(x)\ne x\}$.  A support component of $g$ is a component of $\supp(g)$.
The group $G$ has two Dlab orders, described in Section~\ref{sec:models}.  We equip $G$ with either one of them and fix that choice throughout.

If a group $B$ acts by increasing homeomorphisms on an interval, write $\Fix(B)$ for its common fixed-point set and call a component of the complement of $\Fix(B)$ an \emph{action component} of $B$.  By the \emph{standard action} of $G$ we mean its defining action on $(0,1)$.

We use the notation for the six larger groups from \cite{Med}.  On $I$, $D_H(I)$ consists of the elements that are the identity near both endpoints, $D_{H*}(I)$ of those that are the identity near $0$, $D_{*H}(I)$ of those that are the identity near $1$, and $\overline D_H(I)$ is the full interval group.  On $\overline{\mathbb R}$, $D_H$ consists of the elements with bounded support and $D_{H*}$ of those whose support is bounded below.

We construct an increasing homeomorphism $h$ of $[0,1]$ which has finitely many affine pieces, all with slopes in $\langle2\rangle$, on every compact subinterval of $(0,1]$.  The support components of $h$ form a strictly descending sequence toward $0$.  Conjugation by $h$ normalizes $G$ and preserves its order.  By contrast, for each element of the six larger Dlab groups considered here, the components of its support are well ordered from left to right.

For an arbitrary order-preserving embedding $e:G\hookrightarrow A$, we study the components of the complement of the common fixed-point set of $e(G)$.  They have a least member, denoted later by $J_0$.  Subgroups supported on disjoint source intervals commute, and their least nontrivial components inside $J_0$ are disjoint and occur in the same order as the source intervals.  The resulting nested intervals determine the points of $(0,1)$ and yield a conjugacy between the standard action of $G$ and the action of $e(G)$ on $J_0$.

The main result is the following.

\begin{theorem}\label{thm:main}
Equip $G=D_{\langle2\rangle}(I)$ with either of its two Dlab orders.  There is an order automorphism $\alpha$ of $G$ such that, for every rank-one subgroup $H\leq\mathbb R_{>0}^{\times}$, every one of the six corresponding Dlab groups $A$ listed below, equipped with any of its linear orders, every order-preserving embedding $e:G\hookrightarrow A$, and every $u\in A$, there exists $f\in G$ for which
\[
e(\alpha(f))\neq u^{-1}e(f)u.
\]
Here $A$ may be any of the following six groups
\[
D_H(I),\quad D_{H*}(I),\quad D_{*H}(I),\quad \overline D_H(I),\quad D_H,\quad D_{H*}.
\]
\end{theorem}

The theorem is restricted to the six rank-one groups listed in \cite{Med}.  Zenkov and Medvedev \cite{ZM} also discuss the reflected extended-real group $D_{*H}$; it is not among these six and is not included in the theorem.  Thus the theorem concerns exactly the six groups displayed above.

Section~\ref{sec:models} recalls the order and normal-subgroup facts needed for the six groups.  Section~\ref{sec:h} constructs the automorphism.  Section~\ref{sec:reconstruct} identifies the action on the least component, and Section~\ref{sec:exclude} proves that no element of an ambient Dlab group induces this automorphism by conjugation.

\section{The six Dlab groups and the least action component}\label{sec:models}
We use the usual order topology on the unit interval or on
\(\overline{\mathbb R}=\{-\infty\}\cup\mathbb R\cup\{+\infty\}\), according to the
group.  Thus every action component considered below is an open interval in a
separable linear continuum; if one of its endpoints is infinite, the other
arguments are understood in the extended order.  For a subgroup
\(K\leq\mathbb R_{>0}^{\times}\), write \(D_K([0,1])\) for Dlab's compact-support
group of increasing locally right \(K\)-linear homeomorphisms of \([0,1]\).
For a nonidentity element \(x\) of a linearly ordered group, write \(\operatorname{sgn}(x)=+\)
if \(x>1\) and \(\operatorname{sgn}(x)=-\) if \(x<1\).

We shall use two facts from the definitions of these Dlab groups.  For a nonidentity element, the \emph{first breakpoint} means the first point of its breakpoint sequence, equivalently the left endpoint of its first support component.  First, the
breakpoint sequence of every element of these groups is well ordered from left to right;
hence the nonempty family of its support components has a least member.  Second,
at every finite fixed point an element is affine on some right neighborhood,
with slope in the prescribed subgroup of \(\mathbb R_{>0}^{\times}\).  These properties follow from the definitions in \cite[Sec.~2]{Dlab} and \cite[Sec.~1--2]{ZM}.

Dlab \cite[Theorems~4.1 and~4.3]{Dlab} proved that \(D_K([0,1])\) is algebraically simple and that its linear
orders correspond to the linear orders of \(K\).
For rank-one \(K\), there are exactly two such orders, and the sign of a
nonidentity element is determined by whether the right slope at its first
breakpoint is greater or less than \(1\) in the chosen order of \(K\).
The analogous two-order statements for the one-sided interval and extended-real groups are proved in \cite[Theorems~1.1 and~2.1]{ZM}.  For the compact-support group $D_H$ on the extended real line, the same first-breakpoint description is given in \cite[Sec.~1]{Med}.

\begin{proposition}\label{prop:interfaces}
Let \(H\leq\mathbb R_{>0}^{\times}\) have rank one, let \(A\) be one of
\[
D_H(I),\quad D_{H*}(I),\quad D_{*H}(I),\quad \overline D_H(I),
\quad D_H,\quad D_{H*},
\]
and let \(P\leq A\) be perfect.  For every linear order on \(A\), there is
a normal subgroup \(N\triangleleft A\) containing \(P\) such that the order
induced on \(N\) determines the sign of every nonidentity element of \(P\) from its
leftmost support component and the right slope at the first breakpoint of that
component.  More precisely:
\begin{enumerate}[label=(\roman*)]
\item for \(A=D_H(I)\), take \(N=D_H(I)\).  This is the compact-support group; by \cite[Theorem~4.1]{Dlab}, for rank-one \(H\)
      its two linear orders are exactly the two first-breakpoint orders;
\item for \(A=D_{H*}(I)\), take \(N=D_{H*}(I)\); the proof of
      \cite[Theorem~1.1]{ZM} shows that its two orders are exactly the two
      first-breakpoint orders;
\item for \(A=D_{*H}(I)\), the normal-structure theorem in \cite[Theorem~4.3]{Dlab} gives
      \(D_{*H}(I)/D_H(I)\cong H\).  Hence perfection gives
      \(P\leq D_H(I)\), and \cite[Theorem~4.1]{Dlab} applies to the
      restricted order;
\item for \(A=\overline D_H(I)\), the normal-structure theorem in \cite[Theorem~4.3]{Dlab} gives
      \(\overline D_H(I)/D_{H*}(I)\cong H\).  Hence
      \(P\leq D_{H*}(I)\), and then \cite[Theorem~1.1]{ZM} applies;
\item for \(A=D_{H*}\), take \(N=D_{H*}\) and use
      \cite[Theorem~2.1]{ZM};
\item for \(A=D_H\), take \(N=D_H\).  This compact-support extended-real group has exactly the two first-breakpoint orders by \cite[Sec.~1]{Med}.
\end{enumerate}
In cases (iii) and (iv), the order on the indicated normal subgroup is simply the restriction of the given order on $A$.  No additional choice of order is made.  In case (vi) no passage to a smaller normal subgroup is needed.
\end{proposition}

\begin{proof}
Only (iii) and (iv) require passing to a proper normal subgroup.  The full normal-structure theorem in \cite[Theorem~4.3]{Dlab} gives
\[
 D_{*H}(I)/D_H(I)\cong H,
 \qquad
 \overline D_H(I)/D_{H*}(I)\cong H.
\]
Since \(H\) is abelian, a perfect subgroup has trivial image in either
quotient.  Thus \(P\) lies in the stated kernel.  The remaining assertions
are precisely the cited order classifications.  In particular, replacing the
order of the rank-one slope group by its inverse reverses the two signs but
does not change the first-breakpoint rule.
\end{proof}

For an interval $X$, let $\Homeo^+(X)$ denote the group of increasing homeomorphisms of $X$.  For subsets $E,F$ of a linearly ordered set, write $E<F$ if every point of $E$ precedes every point of $F$.  We shall use the following elementary facts.

\begin{lemma}\label{lem:components}
Let \(X\) be an interval, and let \(B,C\leq\Homeo^+(X)\).
\begin{enumerate}[label=(\roman*)]
\item If \(c\in\Homeo^+(X)\) normalizes \(B\), then \(c\) preserves
      \(\Fix(B)\) and permutes the components of \(X\setminus\Fix(B)\) in
      their natural order.  Consequently, if those components have a least
      member, \(c\) fixes it setwise.
\item If \(B\leq C\), every action component of \(B\) is contained in an
      action component of \(C\).
\item If \(J\) is an action component of \(B\), then a nonempty interval
      \(V\subseteq J\) which is invariant under all of \(B\) must equal \(J\).
\end{enumerate}
\end{lemma}

\begin{proof}
For (i), if \(x\in\Fix(B)\) and \(b\in B\), then
\(b(c x)=c(c^{-1}bc)x=cx\); applying the same argument to \(c^{-1}\) gives
\(c\Fix(B)=\Fix(B)\).  An increasing homeomorphism therefore permutes the
complementary components in order, and an order automorphism of a linearly
ordered set fixes its least element.  For (ii),
\(\Fix(C)\subseteq\Fix(B)\), so each component of the complement of
\(\Fix(B)\) lies in one component of the complement of \(\Fix(C)\).  For
(iii), if \(V\ne J\), at least one endpoint of \(V\) lies in the interior of
\(J\).  Every increasing homeomorphism preserving \(V\) fixes that endpoint,
contrary to the definition of \(J\).
\end{proof}

\begin{proposition}\label{prop:J0}
Let \(G=D_{\langle2\rangle}([0,1])\), let \(H\leq\mathbb R_{>0}^{\times}\)
have rank one, and let \(e:G\hookrightarrow A\) be an order-preserving
embedding into one of the six groups above.  Then the nontrivial action
components of \(e(G)\) have a least member \(J_0\).  The restriction of \(e(G)\)
to every nontrivial action component is faithful, and for each \(1\ne f\in G\)
the leftmost support component of \(e(f)\) lies in \(J_0\); its first right
slope gives exactly the sign of \(f\).
\end{proposition}

\begin{proof}
The group \(G\) is nonabelian simple, hence perfect.  Proposition~\ref{prop:interfaces}
places \(e(G)\), when necessary, in a normal kernel on which the ambient order
has the first-breakpoint description.

Let \(F=\Fix(e(G))\), and let \(J\) be a component of the complement of \(F\).
The kernel of the restriction homomorphism \(G\to\Homeo^+(J)\) is normal in
\(G\).  The action on \(J\) is nontrivial by definition, so simplicity makes
this kernel trivial.  Thus every nonidentity element of \(G\) moves a point in
every action component.

Fix \(1\ne b\in G\).  Each action component \(J\) therefore contains a support
component of \(e(b)\): if \(x\in J\) is moved by \(e(b)\), the support component
of \(e(b)\) containing \(x\) cannot cross an endpoint of \(J\), since such an
endpoint belongs to \(F\).  Choose, in each \(J\), the first support component
of \(e(b)\) lying in \(J\).  If \(J_1<J_2\), the chosen component in \(J_1\)
precedes the one in \(J_2\).  Since the support components of the
element \(e(b)\) are well ordered, the action components themselves have a
least member \(J_0\).

Now let \(1\ne f\in G\).  Faithfulness on \(J_0\) implies that \(e(f)\) has
support in \(J_0\).  No support can occur to the left of \(J_0\): outside the
action components every point is fixed by \(e(G)\), and by minimality there is
no earlier action component.  Hence the global leftmost support component of
\(e(f)\) lies in \(J_0\).  By Proposition~\ref{prop:interfaces}, the ambient
sign is determined by the first right slope there, and order preservation identifies it
with the sign of \(f\) in \(G\).
\end{proof}

\begin{remark}
The rank-one hypothesis on the ambient slope group is essential for the order
classification used above.  No assertion for arbitrary higher-rank slope
groups is made here.
\end{remark}

\section{The alternating-slope automorphism}\label{sec:h}
Set
\[
x_n=2^{-(n+1)}\qquad(n\geq0).
\]
Define $h:[0,1]\to[0,1]$ to be the identity on $[1/2,1]$.  For each $m\geq0$, prescribe slope $2$ on $[x_{2m+2},x_{2m+1}]$ and slope $1/2$ on $[x_{2m+1},x_{2m}]$, joining the pieces continuously from the right.

\begin{lemma}\label{lem:h}
The map $h$ is an increasing homeomorphism of $[0,1]$, it fixes every $x_{2m}$, and its support components are
\[
C_m=(x_{2m+2},x_{2m})\qquad(m\geq0).
\]
In particular $C_{m+1}<C_m$, so the family of support components of $h$ has no least member.
\end{lemma}

\begin{proof}
On the block $[x_{2m+2},x_{2m}]$, the prescribed image length is
\[
2(x_{2m+1}-x_{2m+2})+\frac12(x_{2m}-x_{2m+1})
=x_{2m}-x_{2m+2},
\]
using $x_{n+1}=x_n/2$.  Since $h(x_0)=x_0$, induction from right to left gives $h(x_{2m})=x_{2m}$ for all $m$.  All slopes are positive, and the fixed points $x_{2m}$ tend to $0$, so $h$ is an increasing homeomorphism fixing both endpoints.

On the lower half of each block the derivative of $h(x)-x$ is $1$, while on the upper half it is $-1/2$.  The difference vanishes at the two block endpoints and is positive in between.  Hence the support components are exactly the $C_m$.
\end{proof}

\begin{proposition}\label{prop:alpha}
The formula
\[
\alpha_h(f)=h^{-1}fh
\]
defines an order automorphism of \(G\).
\end{proposition}

\begin{proof}
Every \(f\in G\) is the identity on neighborhoods of \(0\) and \(1\).  Its
support is therefore contained in a compact subinterval of \((0,1)\), on which
\(h\) and \(h^{-1}\) have only finitely many affine pieces, all with slopes in
\(\langle2\rangle\).  Hence \(h^{-1}fh\in G\), and the same argument with
\(h^{-1}\) gives \(h^{-1}Gh=G\).

Conjugation by an increasing homeomorphism carries support components exactly
to their inverse images:
\[
 \supp(h^{-1}fh)=h^{-1}(\supp(f)).
\]
Thus it preserves their left-to-right order and cannot create a new support
component to the left of the old first one.  Let \(a\) be the left endpoint of
a support component of \(f\), let \(a'=h^{-1}(a)\), and let \(k\) be the right
slope of \(f\) at \(a\).  Since \(a'>0\), there are \(c\in\langle2\rangle\) and
\(\varepsilon>0\) such that
\[
 h(a'+t)=a+ct\qquad(0\le t<\varepsilon).
\]
After decreasing \(\varepsilon\), the right germ of \(f\) at \(a\) is
\(f(a+s)=a+ks\), and the corresponding affine piece of \(h^{-1}\) has slope
\(c^{-1}\).  Therefore
\[
 (h^{-1}fh)(a'+t)=a'+kt
\]
for all sufficiently small \(t>0\).  The first breakpoint and its right slope
are consequently transported to \(a'\) with the slope unchanged.  Since both
Dlab orders on \(G\) are determined by that first slope, \(\alpha_h\) preserves
the chosen order.
\end{proof}

\section{The action on the least component}\label{sec:reconstruct}
Fix an order-preserving embedding \(e:G\hookrightarrow A\) and the least action
component \(J_0\) from Proposition~\ref{prop:J0}.  For a relatively compact open interval \(L\) in \((0,1)\), written \(L\Subset(0,1)\), set
\[
G(L)=\{g\in G:\supp(g)\Subset L\}.
\]
After an affine change of coordinates, \(G(L)\) is a compact-support Dlab group
with slope group \(\langle2\rangle\); in particular it is nonabelian simple and
perfect.

The restriction of \(e(G(L))\) to any of its nontrivial action components is
faithful: the restriction kernel is normal in the simple group \(G(L)\), and the
action on such a component is nontrivial.
Moreover, \(e(G(L))\) acts nontrivially somewhere in \(J_0\): choose
\(1\ne f\in G(L)\) and use faithfulness of \(e(G)\) on \(J_0\).  Hence
\(e(G(L))\) has at least one action component in \(J_0\).  Fixing such an
\(f\), faithfulness on every one of these components implies that each contains
a support component of \(e(f)\).  The well ordering of the support components
of \(e(f)\) therefore gives a least \(e(G(L))\)-action component in \(J_0\),
which we denote by \(K(L)\).

\begin{lemma}\label{lem:localbreak}
Let \(L\Subset(0,1)\), put \(P=e(G(L))\), and let \(a\) be the left endpoint of
\(K(L)\) in the ambient ordered interval.  For \(1\ne p\in P\), let \(\lambda(p)\)
be the left endpoint of the first support component of \(p\) in \(K(L)\).
Then \(\lambda(p)\) is the global first breakpoint of \(p\), and
\begin{equation}\label{eq:infimum-one}
 \lambda(p)>a\quad(1\ne p\in P),
 \qquad
 \inf_{1\ne p\in P}\lambda(p)=a.
\end{equation}
\end{lemma}

\begin{proof}
There is no \(e(G)\)-action component to the left of \(J_0\), and points outside
the action components are fixed by \(e(G)\).  Inside \(J_0\), there is no
\(P\)-action component to the left of \(K(L)\).  Hence every point preceding
\(K(L)\) is fixed by \(P\), so the first support component of a nonidentity
\(p\in P\) occurring in \(K(L)\) is its global first support component.

If \(a\) is finite, every element of \(P\) fixes \(a\).  The right germ at \(a\), meaning the restriction to a sufficiently small right
neighborhood of \(a\), is affine.  Taking its slope gives a homomorphism
\[
 \gamma_a:P\longrightarrow H,
\]
where \(\gamma_a(p)\) is the slope of the right germ at \(a\).  Since \(P\) is
perfect and \(H\) is abelian, \(\gamma_a(P)=1\).  A linear germ of slope \(1\)
through the fixed point \(a\) is the identity germ, so each \(p\) fixes a right
neighborhood of \(a\), proving \(\lambda(p)>a\).  If \(a=-\infty\), then the ambient interval is \(\overline{\mathbb R}\).
Only the extended-real groups can occur.  Proposition~\ref{prop:interfaces}
places \(P\) in \(D_H\) or \(D_{H*}\).  By the defining condition in
Zenkov--Medvedev, every element of \(D_{H*}\) fixes an initial interval
\(( -\infty,c]\) for some finite \(c\), and \(D_H\le D_{H*}\); see
\cite[Sec.~2]{ZM}.  Thus every nonidentity \(p\in P\) has a finite first
breakpoint, so \(\lambda(p)>-\infty=a\).

Finally, if the infimum in \eqref{eq:infimum-one} were \(c>a\), then every
element of \(P\) would fix the nonempty initial interval \((a,c)\) of \(K(L)\),
contrary to the definition of an action component.  Thus the infimum is \(a\).
\end{proof}

\begin{lemma}\label{lem:sep}
If \(L,M\Subset(0,1)\) and \(L<M\), then \(K(L)\) and \(K(M)\) are disjoint.
\end{lemma}

\begin{proof}
Put \(P=e(G(L))\) and \(Q=e(G(M))\).  These groups commute.  Hence every
element of \(Q\) normalizes \(P\), and Lemma~\ref{lem:components}(i) shows that
it fixes the least \(P\)-component \(K(L)\) setwise.  Similarly, \(P\) fixes
\(K(M)\) setwise.

Suppose the two intervals meet.  Their intersection is a nonempty interval
invariant under both groups.  Lemma~\ref{lem:components}(iii), first for \(P\)
and then for \(Q\), gives
\[
 K(L)=K(M)=:K.
\]
Let \(a\) be its left endpoint.  For \(1\ne p\in P\) and \(1\ne q\in Q\),
Lemma~\ref{lem:localbreak} gives first breakpoints \(\lambda(p),\lambda(q)>a\)
and
\begin{equation}\label{eq:two-infima}
 \inf_{1\ne p\in P}\lambda(p)=a=
 \inf_{1\ne q\in Q}\lambda(q).
\end{equation}

Choose \(p=e(f)\ne1\) in \(P\) and \(q=e(g)\ne1\) in \(Q\).  Since \(L<M\),
the source element \(fg^n\) has the same sign as \(f\) for every
\(n\in\mathbb Z\): for \(n\ne0\), its first support component lies in \(L\),
and \(n=0\) is immediate.  Thus
\begin{equation}\label{eq:signpowers}
 \operatorname{sgn}(pq^n)=\operatorname{sgn}(p)
 \qquad(n\in\mathbb Z).
\end{equation}
We claim that \(\lambda(p)<\lambda(q)\).  If
\(\lambda(q)<\lambda(p)\), then \(p\) is the identity on a right neighborhood
of \(\lambda(q)\), so the first germ of \(pq^n\) there is exactly the first
germ of \(q^n\).  The elements \(q\) and \(q^{-1}\) have opposite signs; taking
\(n=1\) or \(-1\) contradicts \eqref{eq:signpowers}.

It remains to exclude \(\lambda(p)=\lambda(q)=c\).  Let \(s\ne1\) and
\(t\ne1\) be the right slopes of the first germs of \(p\) and \(q\) at \(c\).
On a common sufficiently small right neighborhood of \(c\), both maps are
linear and fix \(c\), so the first germ of \(pq^n\) has slope \(st^n\).  The
set \(\{st^n:n\in\mathbb Z\}\) has members on both sides of \(1\): if \(t>1\),
then \(st^n\to\infty\) as \(n\to+\infty\) and \(st^n\to0\) as
\(n\to-\infty\), and the case \(t<1\) is reversed.  At most one integer can
satisfy \(st^n=1\), so we may choose \(n\) on the side opposite to \(s\) with
\(st^n\ne1\).  For this \(n\) there is no cancellation of the first germ:
\(c\) remains the first breakpoint of \(pq^n\), and Proposition~\ref{prop:J0} then gives \(pq^n\) a sign opposite to that of \(p\).  This again
contradicts \eqref{eq:signpowers}.  Passing to the other rank-one order reverses both slope comparisons, so the same conclusion holds for either Dlab order.
Consequently
\begin{equation}\label{eq:lambda-order}
 \lambda(p)<\lambda(q)
 \qquad(1\ne p\in P,\ 1\ne q\in Q).
\end{equation}
Fixing \(1\ne p_0\in P\), equation \eqref{eq:lambda-order} gives
\(\lambda(q)>\lambda(p_0)>a\) for every \(1\ne q\in Q\), contradicting the
second equality in \eqref{eq:two-infima}.  Hence \(K(L)\) and \(K(M)\) are
disjoint.
\end{proof}

\begin{lemma}\label{lem:K}
For relatively compact open intervals \(L,M\Subset(0,1)\) the following hold.
\begin{enumerate}[label=(\roman*)]
\item If \(L\subset M\), then \(K(L)\subset K(M)\).
\item If \(L<M\), then \(K(L)<K(M)\).
\item For every \(g\in G\),
      \(e(g)K(L)=K(g(L))\).
\end{enumerate}
Moreover,
\[
 \bigcup_{L\Subset(0,1)}K(L)=J_0.
\]
\end{lemma}

\begin{proof}
Suppose \(L\subset M\).  Since \(G(L)\leq G(M)\),
Lemma~\ref{lem:components}(ii) says that every \(e(G(L))\)-action component is
contained in an \(e(G(M))\)-action component.  The restriction of
\(e(G(M))\) to \(K(M)\) is faithful.  Hence any fixed nonidentity element of
\(e(G(L))\) acts nontrivially on \(K(M)\).  Thus some point of \(K(M)\) is moved by \(e(G(L))\).  The
\(e(G(L))\)-action component containing that point lies entirely in \(K(M)\),
because the endpoints of \(K(M)\) are fixed by \(e(G(M))\), hence by
\(e(G(L))\).  Denote this component by \(C\).  If an \(e(G(L))\)-component preceded
\(K(M)\), Lemma~\ref{lem:components}(ii) would place it in an
\(e(G(M))\)-component preceding \(K(M)\), impossible.  Therefore \(C\) is the
least \(e(G(L))\)-component, namely \(K(L)\).  This proves (i).

Now let \(L<M\).  Lemma~\ref{lem:sep} makes \(K(L)\) and \(K(M)\) disjoint.
Choose \(1\ne f\in G(L)\) positive and \(1\ne g\in G(M)\) negative.  Then
\(fg>1\), because its first source support component is the one in \(L\).  If
\(K(M)<K(L)\), minimality of \(K(L)\) implies that \(e(G(L))\) fixes
\(K(M)\) pointwise.  On the other hand, faithfulness of \(e(G(M))\) on
\(K(M)\) makes the first support component of \(e(g)\) in \(J_0\) lie in
\(K(M)\).  Hence the first support component of \(e(fg)\) in \(J_0\) is that
of \(e(g)\), so Proposition~\ref{prop:J0} makes \(e(fg)<1\), contradicting
order preservation.  Thus \(K(L)<K(M)\), proving (ii).

For (iii), conjugation gives \(gG(L)g^{-1}=G(g(L))\), and \(g(L)\Subset(0,1)\)
because \(g\) is a homeomorphism of \((0,1)\).  Thus conjugation by \(e(g)\) sends \(e(G(L))\) onto \(e(G(g(L)))\), and \(e(g)\) sends their action components order-preservingly onto one another.  Indeed, if \(J\) is a component of the complement of
\(\Fix(e(G(L)))\), then \(e(g)J\) is a component of the complement of
\(\Fix(e(G(g(L))))\), because conjugation by \(e(g)\) carries the two
fixed sets onto one another.  Since \(e(g)\) is increasing, this correspondence
preserves the order of components and therefore sends the least component to
the least component.  Hence \(e(g)K(L)=K(g(L))\).

Put \(U=\bigcup_L K(L)\).  It is nonempty and open.  If \(y,z\in U\), choose
\(L,M\) with \(y\in K(L)\) and \(z\in K(M)\), and then choose
\(N\Subset(0,1)\) containing \(L\cup M\).  Part (i) gives
\(K(L),K(M)\subset K(N)\), so the interval between \(y\) and \(z\) lies in
\(K(N)\).  Thus \(U\) is an interval.  Part (iii) makes it \(e(G)\)-invariant.
If \(U\ne J_0\), then because both are nonempty open intervals and
\(U\subset J_0\), at least one endpoint of \(U\) is a finite point in the
interior of \(J_0\).  Every increasing homeomorphism preserving \(U\) setwise
fixes that endpoint, producing a common fixed point of \(e(G)\) in \(J_0\), a
contradiction.  Hence \(U=J_0\).
\end{proof}

We shall also use the following elementary transitivity fact.

\begin{lemma}\label{lem:transitive}
The group $G$ acts transitively on $(0,1)$.
\end{lemma}

\begin{proof}
It is enough to move a point a sufficiently small distance to the right, since finitely many such moves and their inverses connect any two points of $(0,1)$.  Let $x<y$ satisfy
\[
y<2x,\qquad 2y-x<1.
\]
Put $a=2x-y$ and $d=2y-x$.  Define an increasing piecewise linear homeomorphism to be the identity off $[a,d]$, to have slope $2$ on $[a,x]$, and slope $1/2$ on $[x,d]$.  The two pieces join continuously and return to the identity at $d$, while
\[
f(x)=a+2(x-a)=y.
\]
Thus $f\in G$ and sends $x$ to $y$.  For arbitrary $0<x<y<1$, choose a partition $x=t_0<t_1<\cdots<t_n=y$ with mesh smaller than $\min\{x/2,(1-y)/2\}$.  Then $t_{i+1}<2t_i$ and $2t_{i+1}-t_i<1$ for every $i$, so the preceding construction sends $t_i$ to $t_{i+1}$.  Composing these finitely many bumps sends $x$ to $y$; leftward motion follows by inversion.
\end{proof}

For an open interval \(I\) in the ambient ordered interval, write
\(\ell_-(I)\) and \(\ell_+(I)\) for its left and right endpoints.  Fix
\(x\in(0,1)\) and define
\begin{equation}\label{eq:cuts}
 a_x=\sup\{\ell_+(K(L)):L\Subset(0,x)\},
 \qquad
 b_x=\inf\{\ell_-(K(R)):R\Subset(x,1)\}.
\end{equation}
Both sets are nonempty.  By Lemma~\ref{lem:K}(ii), every interval $K(L)$ used
in the first set lies strictly to the left of every interval $K(R)$ used in the
second.  Fix \(L_0\Subset(0,x)\) and \(R_0\Subset(x,1)\).  Then
\[
 \ell_+(K(L))\leq \ell_-(K(R_0)),\qquad
 \ell_-(K(R))\geq \ell_+(K(L_0))
\]
for all admissible \(L,R\).  The two displayed bounding points are finite:
\(K(L_0)<K(R_0)\) are disjoint open intervals inside \(J_0\), so their adjacent
endpoints are finite real points of the ambient interval.  Completeness of
\(\mathbb R\) therefore gives finite \(a_x,b_x\in J_0\) with \(a_x\leq b_x\).
Put
\[
 F_x=[a_x,b_x]\subset J_0.
\]

\begin{lemma}\label{lem:cuts}
If \(x<z\), then \(F_x<F_z\).  Moreover
\begin{equation}\label{eq:cut-covariance}
 e(g)F_x=F_{g(x)}\qquad(g\in G).
\end{equation}
Consequently every \(F_x\) is a singleton.
\end{lemma}

\begin{proof}
Choose \(M\Subset(0,1)\) with \(x<M<z\).  Since \(M\) occurs on the right of
\(x\) and on the left of \(z\), definition \eqref{eq:cuts} gives
\[
 b_x\leq \ell_-(K(M))<\ell_+(K(M))\leq a_z.
\]
Thus \(F_x<F_z\).

For the second assertion, \(g\) bijects the relatively compact intervals in \((0,x)\)
with those in \((0,g(x))\), and similarly on the right.  By
Lemma~\ref{lem:K}(iii), \(e(g)K(L)=K(g(L))\).  An increasing homeomorphism of an
interval preserves endpoints and existing suprema and infima.  Applying
\(e(g)\) to the two cuts in \eqref{eq:cuts} therefore gives
\(e(g)a_x=a_{g(x)}\) and \(e(g)b_x=b_{g(x)}\), proving
\eqref{eq:cut-covariance}.

The intervals \(F_x\) are pairwise disjoint.  Suppose one, say \(F_x\), is
nondegenerate.  Its interior is then a nonempty open interval.  By
Lemma~\ref{lem:transitive} and \eqref{eq:cut-covariance}, every \(F_z\) is the
image of \(F_x\) under an increasing homeomorphism and is therefore
nondegenerate.  This yields uncountably many pairwise disjoint nonempty open
subintervals of \(J_0\).  But \(J_0\), being an interval in
\(\overline{\mathbb R}\), has a countable dense subset, and each member of a
pairwise disjoint family of nonempty open intervals contains a distinct point
of that subset.  This is impossible.  Hence every \(F_x\) is a singleton.
\end{proof}

Let \(\varphi(x)\) be the unique point of \(F_x\).

\begin{proposition}\label{prop:spatial}
The map
\[
 \varphi:(0,1)\longrightarrow J_0
\]
is an increasing homeomorphism satisfying
\begin{equation}\label{eq:phi-covariance}
 \varphi(gx)=e(g)\varphi(x)\qquad(g\in G,\ x\in(0,1)).
\end{equation}
Equivalently, with \(\pi=\varphi^{-1}\),
\[
 e(g)|_{J_0}=\pi^{-1}g\pi\qquad(g\in G).
\]
\end{proposition}

\begin{proof}
Lemma~\ref{lem:cuts} shows that \(x<z\) implies
\(\varphi(x)<\varphi(z)\), so \(\varphi\) is strictly increasing, and
\eqref{eq:cut-covariance} gives \eqref{eq:phi-covariance}.

A monotone map between real intervals can fail to be continuous only by a jump.  Suppose, for example,
that \(\varphi\) has a nonempty left jump interval at \(x\).  For any \(g\in G\),
continuity and monotonicity of the already given homeomorphism \(e(g)\), together
with \eqref{eq:phi-covariance}, carry that jump interval homeomorphically to the
left jump interval at \(g(x)\).  The same holds for right jumps.  By
Lemma~\ref{lem:transitive}, one jump would therefore produce uncountably many
pairwise disjoint nonempty open intervals in \(J_0\), contradicting the
separability argument used in Lemma~\ref{lem:cuts}.  Thus \(\varphi\) is
continuous.

A continuous strictly increasing map has interval image.  Write
\[
 \alpha=\lim_{x\downarrow0}\varphi(x),\qquad
 \beta=\lim_{x\uparrow1}\varphi(x)
\]
in the extended closure of \(J_0\).  Then
\(\varphi((0,1))=(\alpha,\beta)\), an open subinterval of \(J_0\).
Equation \eqref{eq:phi-covariance} makes this image \(e(G)\)-invariant.  If it
were proper, one of \(\alpha,\beta\) would be a finite point in the interior of
\(J_0\).  Every increasing homeomorphism preserving \((\alpha,\beta)\) setwise
fixes such an endpoint, giving a common fixed point of \(e(G)\) inside the
action component \(J_0\), a contradiction.  Therefore the image is all of
\(J_0\).  A continuous strictly increasing bijection between intervals is a
homeomorphism, and \eqref{eq:phi-covariance} gives the displayed conjugacy.
\end{proof}

\section{Excluding conjugation}\label{sec:exclude}

\begin{lemma}\label{lem:centralizer}
The centralizer of the standard action of \(G\) in
\(\Homeo^+((0,1))\) is trivial.
\end{lemma}

\begin{proof}
Let \(c\) commute with \(G\), and suppose \(c(x)\ne x\).  Replacing \(c\) by
\(c^{-1}\) if necessary, assume \(x<c(x)\).  Choose
\(x<y<c(x)\).  By continuity of \(c\), a sufficiently small relatively compact
open interval \(U\ni x\) satisfies \(\sup U<\inf c(U)\), hence
\(U\cap c(U)=\varnothing\).  The two-slope bump construction from
Lemma~\ref{lem:transitive}, with all breakpoints in \(U\), gives
\(1\ne g\in G\) with \(\supp(g)\Subset U\).  Commutation gives
\[
 c(\supp(g))=\supp(cgc^{-1})=\supp(g),
\]
while the two sides lie in the disjoint intervals \(c(U)\) and \(U\), a
contradiction.  Hence \(c=1\).
\end{proof}

\begin{proposition}\label{prop:exclude}
No \(u\in A\) induces \(\alpha_h\) on \(e(G)\) by conjugation.
\end{proposition}

\begin{proof}
Suppose that
\begin{equation}\label{eq:implement}
 e(\alpha_h(f))=u^{-1}e(f)u\qquad(f\in G).
\end{equation}
Then \(u\) normalizes \(e(G)\).  Lemma~\ref{lem:components}(i) says that \(u\)
permutes the action components of \(e(G)\) in order, so it fixes the least one
\(J_0\) setwise.

Let \(\pi:J_0\to(0,1)\) be the increasing conjugacy from
Proposition~\ref{prop:spatial}, and put
\[
 v=\pi(u|_{J_0})\pi^{-1}.
\]
Transporting \eqref{eq:implement} to \((0,1)\) gives
\[
 v^{-1}fv=h^{-1}fh\qquad(f\in G).
\]
Thus \(vh^{-1}\) centralizes the standard action of \(G\), and
Lemma~\ref{lem:centralizer} yields \(v=h\).

By Lemma~\ref{lem:h}, the support components of \(h\) are
\(C_m=(x_{2m+2},x_{2m})\), with \(C_{m+1}<C_m\) and no least member.  Therefore
\[
 D_m=\pi^{-1}(C_m)
\]
is a component of \(\supp(u)\cap J_0\).  Its two endpoints
\(\pi^{-1}(x_{2m+2})\) and \(\pi^{-1}(x_{2m})\) are finite interior points of
\(J_0\) fixed by \(u\), because \(u|_{J_0}=\pi^{-1}h\pi\).  Hence no connected
component of the open set \(\supp(u)\) can cross either endpoint of \(D_m\).
Thus every \(D_m\) is a full support component of \(u\), not
merely a component of its restriction to \(J_0\).  Since
\(D_{m+1}<D_m\) for all \(m\), these support components form a nonempty family
with no least member, contradicting the well-ordering of support components of
an element of a Dlab group.  This proves the proposition.
\end{proof}

\begin{proof}[Proof of Theorem~\ref{thm:main}]
By Proposition~\ref{prop:alpha}, $\alpha_h$ is an order automorphism of $G$.  Proposition~\ref{prop:spatial} identifies the action of $e(G)$ on $J_0$ with the standard action of $G$, and Proposition~\ref{prop:exclude} shows that no element of any of the six groups induces $\alpha_h$ by conjugation.
\end{proof}

\section*{Report of AI use}
The main argument was generated with the assistance of Albilich, a generative artificial-intelligence assistant for mathematical research developed by the authors. All arguments were subsequently fully understood, completely rewritten, and independently verified by the authors.

\section*{Acknowledgements}
This work was partially supported by a grant from the Simons Foundation (\#918096, to YY).

\section*{Disclosure Statement}
The authors declare that they have no competing interests and no conflicts of interest.

\section*{Data Availability Statement}
Data sharing is not applicable to this article, as no data sets were generated or analysed during the current study.

\end{document}